\documentclass[11pt, notitlepage]{article}
\usepackage{amssymb,amsmath,comment}
\usepackage{amsthm}
\catcode`\@=11 \@addtoreset{equation}{section}
\def\thesection{\arabic{section}}

\def\theequation{\thesection.\arabic{equation}}
\catcode`\@=12
\usepackage{colortbl}
\usepackage{hyperref}
\usepackage[mathscr]{eucal}
\usepackage{epsf}
\usepackage{esint}
\usepackage{a4wide}
\def\R{\mathbb{R}}

\newcommand{\e}{\varepsilon}

\newcommand{\noi} {\noindent}

\usepackage[all]{xy}
\catcode`\@=11
\def\theequation{\@arabic{\c@section}.\@arabic{\c@equation}}
\catcode`\@=12

\newtheorem{Theorem}{Theorem}[section]
\newtheorem{Lemma}[Theorem]{Lemma}
\newtheorem{prop}[Theorem]{Proposition}
\newtheorem{Corollary}[Theorem]{Corollary}
\newtheorem{Remark}[Theorem]{Remark}

\newtheorem*{ack}{Acknowledgements}

\usepackage{todonotes}

\begin{document}

\title{Decay of extremals of Morrey's inequality}
\author{Ryan Hynd, Simon Larson and Erik Lindgren}

\maketitle

\begin{abstract}\noindent
We study the decay (at infinity) of extremals of Morrey's inequality in $\R^n$. These are functions satisfying
$$
\displaystyle \sup_{x\neq y}\frac{|u(x)-u(y)|}{|x-y|^{1-\frac{n}{p}}}= C(p,n)\|\nabla u\|_{L^p(\R^n)} , 
$$
where $p>n$ and $C(p,n)$ is the optimal constant in Morrey's inequality. We prove that if $n \geq 2$ then any extremal has a power decay of order $\beta$ for any
$$
\beta<-\frac13+\frac{2}{3(p-1)}+\sqrt{\left(-\frac13+\frac{2}{3(p-1)}\right)^2+\frac13}.
$$
\end{abstract}

\noi {Keywords: Morrey's inequality, Decay at infinity, The $p$-Laplace equation}

\noi{\textit{2020 Mathematics Subject Classification: 35B65, 35J70}

{\vspace{0.01in}}


\section{Introduction}
Morrey's classical inequality in $\R^n$ states that for $p>n$, there is a constant $C=C(p,n)$ such that
\begin{equation}\label{eq:morrey}
[u]_{C^{0,1-\frac{n}{p}}(\R^n)}=\sup_{ x\neq y}\frac{|u(x)-u(y)|}{|x-y|^{1-\frac{n}{p}}}\leq  C(p,n)\left(\int_{\R^n}|\nabla u|^p dx\right)^\frac1p,
\end{equation}
for all functions whose first order partial derivatives belong to $L^p(\R^n)$. In a series of papers (cf. \cite{HS,HS2,HS3}), Hynd and Seuffert study this inequality and prove that there is a smallest constant $C>0$ such that \eqref{eq:morrey} holds and that there are extremals of this inequality. An extremal is a function for which equality is attained in \eqref{eq:morrey}.
 They also prove that up to translation, rotation, dilatation and multiplication by a constant, any extremal function $u$ satisfies
\begin{enumerate}\label{extremal}
\item $-\Delta_p u =  \delta_{e_n}-\delta_{-e_n}$ in $\R^n$,
\item $|u|\leq 1$,  $u(e_n)=1$, $u(-e_n)=-1$,
\item $u$ is antisymmetric with respect to the $x_n$-variable,
\item $u$ is positive in $\R^n\cap \{x_n>0\}$.

\end{enumerate}
See Theorem 2.4 and Propositions 3.1, 3.4 and 3.5 in \cite{HS}. Here $\Delta_p u:=\nabla\cdot (|\nabla u|^{p-2}\nabla u)$ is the $p$-Laplace operator. In addition to this, they study the behavior at infinity of extremals in dimensions $n\geq 2$ and prove that there is $\beta>0$ and $C>0$ such that 
\begin{equation}
\label{dec}
\sup_{|x|\geq R} |u|\leq CR^{-\beta},\quad \text{for all $R$}.
\end{equation}
See Corollary 4.7 in \cite{HS2}. However, no estimate of $\beta$ is given.

The main objective of this paper is to provide an explicit exponent $\beta$. More precisely, we prove the following theorem.
\begin{Theorem} \label{teo:1} Suppose $p>n\geq 2$, that $u$ is an extremal of \eqref{eq:morrey} satisfying properties (1)-(4) above and
$$\beta<-\frac13+\frac{2}{3(p-1)}+\sqrt{\left(-\frac13+\frac{2}{3(p-1)}\right)^2+\frac13}.
$$
Then there is $C_1=C_1(\beta,p,n)$ such that
$$
|u(x)|\leq C_1|x|^{-\beta},
$$
for all $|x|\geq 1$.
\end{Theorem}
As a corollary, we obtain the corresponding decay for the gradient.
\begin{Corollary}\label{cor:1} Under the assumptions of Theorem \ref{teo:1}, there is $C_2=C_2(\beta,p,n)$ such that
$$
|\nabla u(x)|\leq C_2|x|^{-\beta-1},
$$
for all $|x|\geq 2$.
\end{Corollary}
\begin{Remark}
    A couple of remarks:
    \begin{enumerate}
       
        \item By (2) above the conclusion of Theorem~\ref{teo:1} is valid also for $|x|\leq 1$. However, the same is not true for Corollary~\ref{cor:1}. Indeed, by~\cite[Proposition 2.8]{HL19} $|\nabla u(x)|$ becomes unbounded as $x \to \pm e_n$.

        \item In dimension one, the extremal satisfying (1)-(4) is explicitly given by 
        $$u(x) = 
        \begin{cases} 
            -1  &\mbox{for } x\leq -1,\\
            x  &\mbox{for } x \in (-1, 1),\\
            1  &\mbox{for }x\geq 1.
        \end{cases}$$
        Therefore, the assumption $n \geq 2$ is necessary in Theorem~\ref{teo:1}. However, the bound in Corollary~\ref{cor:1} is trivially true when $n=1$.
    \end{enumerate}
\end{Remark}

Although it is of intrinsic interest to further understand the extremal functions of Morrey's inequality, our motivation for the results in this short note stem from a particular application. Namely, in a forthcoming paper we address the existence of minimizers in a certain variational problem and an estimate for the decay of Morrey extremals and their gradients entered as a key technical ingredient.

\subsection{Known results} The asymptotic behavior at infinity for solutions of PDEs has been studied before. See for instance \cite{Ser65} where it is proved that bounded $p$-harmonic functions in exterior domains has a limit at infinity. Related results can also be found in \cite{KV86}, \cite{FP11} and \cite{FP13}. 

\subsection{Plan of the paper}
In Section \ref{sec:prel}, we discuss notation, definitions and certain prerequisites for this paper. This is followed by Section \ref{sec:aro}, where Aronsson's $p$-harmonic functions obtained through separation of variables are discussed. In Section \ref{sec:slow}, we study the singularities of functions that are $p$-harmonic in punctured domains. Finally, we prove our main results in Section \ref{sec:main}.

\begin{ack}
We wish to thank Peter Lindqvist, who drew our attention to Aronsson's paper \cite{Aro86}.

 R. H. was supported by an American Mathematical Society Claytor-Gilmer Fellowship. E. L. has been supported by the Swedish Research Council, grant no. 2017-03736 and 2016-03639. S. L. was supported by Knut and Alice Wallenberg Foundation grant KAW 2021.0193. 

 Part of this material is based upon work supported by the Swedish Research Council under grant no. 2016-06596 while all three authors were participating in the research program ``Geometric Aspects of Nonlinear Partial Differential Equations'', at Institut Mittag-Leffler in Djursholm, Sweden, during the fall of 2022.

 \end{ack}
 
\section{Preliminaries}\label{sec:prel} 
Throughout the paper we work in $\R^n$ with $p>n \geq 2$ and we will denote the exponent appearing in Theorem \ref{teo:1} by $$\beta_p:=-\frac13+\frac{2}{3(p-1)}+\sqrt{\left(-\frac13+\frac{2}{3(p-1)}\right)^2+\frac13}.$$

We will need a few results regarding $p$-harmonic functions. The following assertion is contained in Theorem 1.1 and Remark 1.6 in \cite{KV86}.
 \begin{Theorem}\label{teo:kv} Suppose that $|u|\leq 1$ in $B_1\setminus \{0\}$, $u\in W_\textup{loc}^{1,p}(B_1\setminus\{0\})$  and that 
 $$
- \Delta_p u = 0\text{ in }B_1\setminus \{0\}.
 $$
 Then $u\in W_\textup{loc}^{1,p}(B_1)$ and there is $\gamma$ such that  
 $$
- \Delta_p u =|\gamma|^{p-2}\gamma \delta_0\quad \text{in $B_1$}.
 $$
 \end{Theorem}
 The next result is Corollary 2.4 in \cite{HS2}.
 \begin{prop}\label{prop:ryan}
Suppose $u$ is bounded and satisfies
$$
- \Delta_p u =c\delta_{x_0}
$$
in $\R^n$ for some point $x_0$ and some constant $c$. Then $u$ is necessarily constant and $c=0$.
\end{prop}
\section{Solutions in the plane by separation of variables}\label{sec:aro}
In \cite{Aro86}, Aronsson studies $p$-harmonic functions for $p>2$ in sectors of $\R^2$ which have the form $u(r, \phi)=r^{-\kappa}f(\phi)$ for $\kappa> 0$\footnote{Note that $\kappa$ here corresponds to $-k$ in Aronsson's notation and therefore the resulting equations differ accordingly. Aronsson considers $k$ of arbitrary sign but here only singular solutions will be important.} and where $(r, \phi)$ are polar coordinates. In Lemma 1 case $\alpha$) in \cite{Aro86}, it is proved that $u$ is $p$-harmonic in the cone $r>0$, $\phi\in I$ if and only if 
\begin{equation}
\label{eq:geq}
g(\phi):=(f'(\phi))^2+\left(1+\frac{1}{a\kappa}\right)\kappa^2(f(\phi))^2>0, \quad a=\frac{p-1}{p-2}
\end{equation}
and there is a constant $C>0$ such that 
\begin{equation}
\label{eq:geq2}
[(f'(\phi))^2+\kappa^2(f(\phi))^2]^{-\kappa} = C^2|g(\phi)|^{-\kappa-1}.
\end{equation}
Recall that $p>n=2$ so $a>0$. On page 145 in \cite{Aro86}, the following semi-explicit formula for a solution is given:
$$
\phi =\theta-a(1+\kappa)\int_0^\theta \frac{1}{\cos^2\theta'+a\kappa} d\theta', \quad f=\left(1+\frac{\cos^2\theta}{a\kappa}\right)^\frac{-\kappa-1}{2}\cos \theta.
$$
In order to see that this implies \eqref{eq:geq} and \eqref{eq:geq2}, it is sufficient to compute $f'(\phi)$ and find that 
$$
f'(\phi)=\kappa\left(1+\frac{\cos^2\theta}{a\kappa}\right)^\frac{-\kappa-1}{2}\sin \theta
$$
so that 
$$
(f'(\phi))^2+\kappa^2(f(\phi))^2 = \kappa^2 \left(1+\frac{\cos^2\theta}{a\kappa}\right)^{-\kappa-1}.
$$
It follows that 
$$
g(\phi)=\kappa^2\left(1+\frac{\cos^2\theta}{a\kappa}\right)^{-\kappa}>0
$$
and that \eqref{eq:geq2} holds with $C=\kappa^2$.

Upon integration, the relation between $\phi$ and $\theta$ simplifies to
$$
\phi = \theta -\left(\frac{1}{\kappa}+1\right)\mu \arctan\left(\mu\tan\theta\right), \quad \mu = \frac{\sqrt{a\kappa}}{\sqrt{a\kappa+1}},
$$
for $\theta \in (-\pi/2,\pi/2)$. This implies that the range of possible $\phi$ is $I=(\phi(\pi/2), \phi(-\pi/2))$, which is an interval of length
$$
\tilde L=\pi\left(\mu\Bigl(1+\frac{1}{\kappa}\Bigr)-1\right).
$$
We also note that $f$ is positive when $\cos\theta$ is positive which is exactly on the interval $I$. Hence, this defines a positive solution of the $p$-Laplace equation in a cone with opening $\tilde L$, which is zero on the boundary rays of the cone.

Since we will be interested in solutions in cones with opening $\pi$ or larger, we let $\tilde L=\pi L$ and obtain  
\begin{equation}
\label{eq:Lk}
L=\frac{\sqrt{a\kappa}}{\sqrt{a\kappa+1}}\Bigl(1+\frac{1}{\kappa}\Bigr)-1,
\end{equation}
which implies
\begin{equation}
\label{eq:Lk2}
(L+1)^2=\frac{(\kappa+1)^2}{\kappa^2+\frac{\kappa}{a}}.
\end{equation}
Upon recalling that $a=(p-1)/(p-2)$, it is clear that $L$ is strictly decreasing in $p$. It is not hard to see that if $L=1$, \eqref{eq:Lk2} gives
$
\kappa=\beta_p.
$
This corresponds to a half plane solution. Here we observe that $\beta_p$ decreases to its limit $1/3$ as $p\to \infty$, hence $\beta_p>1/3$ for all $p> 2$. In addition, differentiating \eqref{eq:Lk} gives
$$
\frac{dL}{d\kappa}=\frac{\sqrt{a}(\kappa(1-2 a)- 1)}{2 \kappa^\frac32(a\kappa+ 1)^\frac32}< 0.
$$
Therefore, $L$ is strictly decreasing in $\kappa$. Hence, for $\delta>0$ there is a one to one correspondance between $L=1+\delta$ and the corresponding power $\kappa(\delta)<\beta_p$. Moreover, $\kappa(\delta)\to \beta_p$ as $\delta\to 0$. Upon renaming $\delta$, we may summarize our conclusions as:

\emph{For any power $\beta<\beta_p$, there is a $\delta>0$ and a $p$-harmonic function $u=r^{-\beta}f(\phi)$ in the cone $r>0$, $\phi\in (-\frac{\pi}{2}-\delta,\frac{\pi}{2}+\delta)$ which is positive for $\phi\in (-\frac{\pi}{2}-\delta,\frac{\pi}{2}+\delta)$ and satisfies $u(r,-\frac{\pi}{2}-\delta)=u(r,\frac{\pi}{2}+\delta)=0$.}

\section{Slow decay implies boundedness} \label{sec:slow}
In this section, we prove that certain global solutions of the $p$-Laplace equation cannot blow up at the origin if they blow up too slowly.
\begin{prop}\label{lem:bounded} Suppose $u\geq 0$, $\Delta_p u =0$ in $\R^n\cap \{x_n>0\}$, $u(x)=0$ for $x_n=0$ except possibly at the origin, $|u(x)|\leq 1$ for $|x|\geq 1$ and that
$$
|u(x)|\leq |x|^{-\beta},
$$
for $|x|\leq 1$ where $\beta<\beta_p$. Then $|u(x)|\leq 1$ for all $x\neq 0$.
\end{prop}
\begin{proof} Take $\tau>0$ so that $\beta+\tau<\beta_p$. From the discussion in Section \ref{sec:aro}, it follows that there is a solution $w$ of the form 
$$
w(r,\phi)=r^{-\beta-\tau }f(\phi)
$$
valid in the cone $r>0, \phi\in (-\frac{\pi}{2}-\delta,\frac{\pi}{2}+\delta)$ for some $\delta>0$. Here the polar coordinates are chosen in the $x_{n-1}x_n$-plane so that $\phi=\pm\pi/2$ corresponds to $x_n=0$ and $\phi=0$ corresponds to the positive $x_n$-axis. Moreover, $f(\phi)>0$ for $\phi\in (-\frac{\pi}{2}-\delta,\frac{\pi}{2}+\delta)$. This means in particular that that $r^{\beta+\tau}w(r,\phi)>c_f>0$ in $B_1\cap \{x_n\geq 0\}\setminus\{0\}$. If we are in dimension three or more, we extend this solution trivially to be a solution in $\R^n\cap \{x_n\geq 0\}$. Thus, 
$$
w(x_1,\ldots,x_n)=(x_{n-1}^2+x_n^2)^{-\beta/2-\tau/2}f(\phi)\geq |x|^{-\beta-\tau}f(\phi).
$$
The idea is to use $w$ as a base for a barrier that will force $u$ to be bounded.

To see this, let $v=1+\e w$ for $\e>0$. We now wish to compare $u$ with $v$ in $B_1\setminus B_\rho\cap \{x_n\geq 0\}$. Take $\rho =\rho(\e)\in (0,1)$ such that $\rho^{-\beta}\leq \e c_f\rho^{-\beta-\tau}$. On $\partial B_1\cap \{x_n>0\}$ we have $v\geq 1\geq u$, on $\partial B_\rho\cap \{x_n\geq 0\}$ we have 
$$
|u(x)|\leq |\rho|^{-\beta}\leq \varepsilon c_f\rho^{-\beta-\tau}\leq \e w\leq v
$$
and on $B_1\setminus B_\rho\cap \{x_n=0\}$ we have $v\geq 0=u$. The comparison principle implies
$
u\leq v$ in $B_1\setminus B_\rho$. Moreover, since $\rho^{-\beta}\leq \e c_f\rho^{-\beta-\tau}$, we trivially have $u\leq |x|^{-\beta}\leq \varepsilon c_f|x|^{-\beta-\tau}\leq v$ in $B_\rho\setminus \{0\}$. We conclude that 
$u\leq v$ in $B_1\setminus \{0\}$. This inequality does not depend on $\varepsilon$ so by letting $\varepsilon\to 0$, we obtain $u\leq 1$ in $B_1\setminus \{0\}$ and the proof is complete.

\bigskip

\end{proof}

\section{Proof of the main theorem}\label{sec:main}
\begin{prop}\label{prop:decay} Assume $\Delta_p u =0$ in $\R^n\setminus \overline B_1$, $|u(x)|\leq 1$ for $|x|\geq 1$, $u\geq 0$ for $x_n\geq 0$ and that $u$ is antisymmetric with respect to the $x_n$-variable. Then for each $\beta<\beta_p$, there is a constant $C=C(n,p,\beta)$ such that
$$
\sup_{|x|\geq r} |u(x)|\leq Cr^{-\beta},\quad r\geq 1.
$$
\end{prop}
We prove the proposition by proving the lemma below.

\begin{Lemma}\label{lem:Sr} Assume the hypotheses of Proposition \ref{prop:decay}. Then for each $\beta<\beta_p$, there is a constant $C=C(n,p,\beta)>0$ such that for all $r\geq 1$ at least one of the following properties hold:
\begin{enumerate}
\item $S_r:=\displaystyle \sup_{|x|\geq r} |u(x)|\leq Cr^{-\beta}, $
\item There is a $k\geq 1$ such that $2^{-k}r\geq 1$ and $S_r\leq 2^{-k\beta}S_{2^{-k}r}.$
\end{enumerate}
\end{Lemma}

We first explain how Proposition \ref{prop:decay} follows from this lemma.
\begin{proof}[Proof of Proposition \ref{prop:decay}] If alternative (1) of Lemma \ref{lem:Sr} holds for all $r\geq 1$, then we are done. If not, we pick an $r$ for which (1) fails so that, by alternative (2), 
$$
S_r\leq 2^{-k_1\beta}S_{2^{-k_1}r},
$$
for some integer $k_1$ with $2^{-k_1}r\geq 1$. If (1) holds for $2^{-k_1}r$, then
$$
S_{r}\leq 2^{-k_1\beta}S_{2^{-k_1}r}\leq 2^{-k_1\beta}C(2^{-k_1}r)^{-\beta}=Cr^{-\beta}
$$
and again we are done. If not, we continue with 
$$
S_{2^{-k_1}r}\leq 2^{-k_2\beta}S_{2^{-k_2}2^{-k_1}r}, 
$$
where $2^{-k_2}2^{-k_1}r\geq 1$. Iterating this as long as alternative (1) fails, we obtain
$$
S_r\leq 2^{-k_n\beta}\cdots 2^{-k_1\beta}S_{2^{-k_n}\cdots 2^{-k_1}r}=2^{-(k_1+\cdots +k_n)\beta}S_{2^{-k_1-\cdots - k_n} r}, 
$$
where $2^{-k_1-\cdots - k_n} r\geq 1$. Since every $k_j\geq 1$, the procedure must stop after a finite number of steps (depending $r$), say after $n$ steps. Then alternative (1) holds for the radius $2^{-k_1-\cdots - k_n} r$ and so, finally, 
$$
S_r\leq 2^{-(k_1+\cdots +k_n)\beta}S_{2^{-k_1-\cdots -k_n} r}\leq 2^{-(k_1+\cdots +k_n)\beta}C(2^{-k_1-\cdots - k_n} r)^{-\beta}\leq Cr^{-\beta}.
$$
This proves the claim.
\end{proof}

\begin{proof}[~Proof of Lemma \ref{lem:Sr}]
We assume towards a contradiction that the statement is false. Then, for each $j=1,2,3,\ldots$, we may find $r_j\geq 1$ such that 
\begin{enumerate}
\item $S_{r_j}\geq jr_j^{-\beta}$, 
\item  $S_{r_j}\geq 2^{-k\beta}S_{2^{-k}r_j}, $
for all $k\geq 1$ such that $2^{-k}r_j\geq 1$.
\end{enumerate}
Note that the point 1) above forces $r_j\to \infty$, since $u$ is bounded. Define 
$$
v_j(x)=\frac{u(r_jx)}{S_{r_j}}.
$$
Setting $S_r(v_j) :=\sup_{|x|\geq r}|v_j(x)|$, it follows that $v_j$ satisfies
\begin{enumerate}
\item[(a)] $S_1(v_j)=1$,
\item[(b)] $S_{2^{-k}}(v_j)\leq 2^{k\beta}$, for all $k$ such that $2^{-k}r_j\geq 1$,
\item[(c)] $\Delta_p v_j = 0$ in $\R^n\setminus \overline B_\frac{1}{r_j}$.
\end{enumerate}
Using local estimates for the $p$-Laplace equation, we may therefore extract a subsequence converging locally uniformly in $\R^n\setminus \{0\}$ to a function $v$. We also note that by Corollary 4.2 in \cite{HS2}, we know that 
$$
S_1(v_j)= \sup_{|x|\geq 1} v_j= \sup_{|x|= 1} v_j.
$$
Therefore, the local uniform convergence assures that $v$ satisfies
\begin{enumerate}
\item[(a')]$ \displaystyle\sup_{|x|= 1} v= 1$,
\item[(b')]$S_{2^{-k}}(v)\leq 2^{k\beta}$, for all $k\geq 1$,
\item[(c')] $\Delta_p v = 0$ in $\R^n\setminus \{0\}$.
\end{enumerate}
We also note that since each $v_j$ is antisymmetric with respect to the $x_n$-variable  and non-negative in $\{x_n\geq 0\}$, so is the limit $v$. By Proposition \ref{lem:bounded}, $|v(x)|\leq 1$ for all $x\neq 0$. We can then apply Theorem \ref{teo:kv} combined with Proposition \ref{prop:ryan} and conclude that $v$ has to be identically zero. This contradicts (a') above.
 \end{proof}

The proof of Theorem \ref{teo:1} is now immediate.
\begin{proof}[~Proof of Theorem \ref{teo:1}]
As mentioned in the introduction, up to translation, rotation and dilatation, any extremal function $u$ is antisymmetric with respect to the $x_n$-variable, positive in $\R^n\cap \{x_n>0\}$, $p$-harmonic outside $\overline B_1$ and satisfies $|u|\leq 1$. Therefore, Proposition \ref{prop:decay} applies and the proof is complete.
\end{proof}

From this and interior estimates for the $p$-Laplace equation, Corollary \ref{cor:1} follows.
\begin{proof}[~Proof of Corollary \ref{cor:1}]
Take $x$ such that $|x|=R\geq 2$ and $\beta<\beta_p$. Then Theorem \ref{teo:1} implies
$$
\sup_{B_{R/4}(x)} |u|\leq CR^{-\beta}.
$$
Since $R\geq 2$, ${B_{R/4}(x)}\cap B_1=\emptyset$ so that $u$ is $p$-harmonic in $B_{R/4}(x)$.  By interior gradient estimates (cf. \cite{Eva82}, \cite{Lew83} or \cite{Ura68})
$$
\sup_{B_{R/8}(x)}|\nabla u|\leq CR^{-1}\sup_{B_{R/4}(x)} |u(x)|\leq CR^{-\beta-1}
$$
and in particular
$$
|\nabla u(x)|\leq CR^{-\beta-1},
$$
which completes the proof of Corollary~\ref{cor:1}.\end{proof}

\medskip
\bibliographystyle{plain}
\bibliography{HMbib}

\begin{thebibliography}{10}

\bibitem{Aro86}
Gunnar Aronsson.
\newblock Construction of singular solutions to the {$p$}-harmonic equation and
  its limit equation for {$p=\infty$}.
\newblock {\em Manuscripta Math.}, 56(2):135--158, 1986.

\bibitem{Eva82}
Lawrence~C. Evans.
\newblock A new proof of local {$C^{1,\alpha }$} regularity for solutions of
  certain degenerate elliptic p.d.e.
\newblock {\em J. Differential Equations}, 45(3):356--373, 1982.

\bibitem{FP11}
Martin Fraas and Yehuda Pinchover.
\newblock Positive {L}iouville theorems and asymptotic behavior for
  {$p$}-{L}aplacian type elliptic equations with a {F}uchsian potential.
\newblock {\em Confluentes Math.}, 3(2):291--323, 2011.

\bibitem{FP13}
Martin Fraas and Yehuda Pinchover.
\newblock Isolated singularities of positive solutions of {$p$}-{L}aplacian
  type equations in {$\Bbb{R}^d$}.
\newblock {\em J. Differential Equations}, 254(3):1097--1119, 2013.

\bibitem{HL19}
Ryan Hynd and Erik Lindgren.
\newblock Extremal functions for {M}orrey's inequality in convex domains.
\newblock {\em Math. Ann.}, 375(3-4):1721--1743, 2019.

\bibitem{HS2}
Ryan Hynd and Francis Seuffert.
\newblock Asymptotic flatness of {M}orrey extremals.
\newblock {\em Calc. Var. Partial Differential Equations}, 59(5):Paper No. 159,
  24, 2020.

\bibitem{HS3}
Ryan Hynd and Francis Seuffert.
\newblock On the symmetry and monotonicity of {M}orrey extremals.
\newblock {\em Commun. Pure Appl. Anal.}, 19(11):5285--5303, 2020.

\bibitem{HS}
Ryan Hynd and Francis Seuffert.
\newblock Extremal functions for {M}orrey's inequality.
\newblock {\em Arch. Ration. Mech. Anal.}, 241(2):903--945, 2021.

\bibitem{KV86}
Satyanad Kichenassamy and Laurent V\'{e}ron.
\newblock Singular solutions of the {$p$}-{L}aplace equation.
\newblock {\em Math. Ann.}, 275(4):599--615, 1986.

\bibitem{Lew83}
John~L. Lewis.
\newblock Regularity of the derivatives of solutions to certain degenerate
  elliptic equations.
\newblock {\em Indiana Univ. Math. J.}, 32(6):849--858, 1983.

\bibitem{Ser65}
James Serrin.
\newblock Singularities of solutions of nonlinear equations.
\newblock In {\em Proc. {S}ympos. {A}ppl. {M}ath., {V}ol. {XVII}}, pages
  68--88. Amer. Math. Soc., Providence, R.I., 1965.

\bibitem{Ura68}
Nina~N. Ural'ceva.
\newblock Degenerate quasilinear elliptic systems.
\newblock {\em Zap. Nau\v{c}n. Sem. Leningrad. Otdel. Mat. Inst. Steklov.
  (LOMI)}, 7:184--222, 1968.

\end{thebibliography}

\noindent {\textsf{Ryan Hynd\\  Department of Mathematics\\ University of Pennsylvania\\209 South 33rd Street\\
Philadelphia, PA 19104-6395}  \\
\textsf{e-mail}: rhynd@math.upenn.edu\\

\noindent {\textsf{Simon Larson\\  Mathematical Sciences\\ Chalmers University of Technology \& the University of Gothenburg\\ 412 96, Gothenburg, Sweden}  \\
\textsf{e-mail}: larsons@chalmers.se\\

\noindent {\textsf{Erik Lindgren\\  Department of Mathematics\\ KTH -- Royal Institute of Technology\\ 100 44, Stockholm, Sweden}  \\
\textsf{e-mail}: eriklin@math.kth.se\\

\end{document}